\documentclass{TCAM}

\usepackage[english]{babel}     

\usepackage[utf8]{inputenc}   

\usepackage{epsfig}
\usepackage{multirow}
\usepackage{array}
\usepackage{url}
\usepackage{graphics}
\usepackage{subfigure}
\usepackage{psfrag}
\usepackage{listingsutf8}  
\lstdefinestyle{mystyle}{
    basicstyle=\ttfamily\footnotesize,
    breakatwhitespace=false,         
    breaklines=true,                 
    captionpos=b,                    
    keepspaces=true,                 
    showspaces=false,                
    showstringspaces=false,
    showtabs=false,                  
    tabsize=2
}

\usepackage[hypertexnames=false,colorlinks=true,breaklinks=true,bookmarks=true,urlcolor=blue,citecolor=blue,linkcolor=blue,bookmarksopen=false,draft=false]{hyperref}
\usepackage{float}
\usepackage{amsmath, amsfonts, amsthm, amssymb}
\usepackage{placeins}
\usepackage{pdfpages}
\usepackage[ruled,vlined,linesnumbered,norelsize]{algorithm2e}
\usepackage{comment}
\usepackage{placeins}
\usepackage{subfigure}
\usepackage{nicematrix} 
\usepackage{tikz}
\usetikzlibrary{fit}
\usepackage{epsfig,graphicx}
\usepackage{multicol,pst-plot}
\usepackage{pstricks}
\usepackage{amsmath}
\usepackage{amsfonts}
\usepackage{amssymb}
\usepackage{eucal}
\usepackage[left=2cm,right=2cm,top=2cm,bottom=2cm]{geometry}
\usepackage{lipsum}
\usepackage{listings}
\usepackage{color}

\theoremstyle{plain}
\newtheorem{theorem}{Theorem}

\newtheorem{corollary}[theorem]{Corollary}
\newtheorem{proposition}[theorem]{Proposition}

\DeclareMathOperator*{\argmin}{argmin}
\DeclareMathOperator{\prox}{\textstyle{prox}}
\DeclareMathOperator{\vvec}{\textstyle{vec}}
\DeclareMathOperator{\rank}{\textstyle{rank}}

\makeatletter
\renewcommand*{\top}{%
  {\mathpalette\@transpose{}}%
}
\newcommand*{\@transpose}[2]{%
  \raisebox{\depth}{$\m@th#1\scriptscriptstyle\mathsf{T}$}%
}
\makeatother

\usepackage[symbol*]{footmisc}
\setfnsymbol{chicago}

\title{Sparse symmetric generalized inverses\\ for sparse symmetric matrices\footnote{A short preliminary version of this work appeared at ERMAC 2026.}}

\runningtitle{Sparse symmetric generalized inverses for sparse symmetric matrices}

\author{ 
Ananias Sousa Machado\aff{
Universidade Federal do Rio de Janeiro, Rio de Janeiro, RJ, Brazil, e-mail: ananiasmachado@poli.ufrj.br
} , 
 Marcia Fampa\aff{Universidade Federal do Rio de Janeiro, Rio de Janeiro, RJ, Brazil, e-mail: fampa@cos.ufrj.br 
} , 
Jon Lee\aff{
University of Michigan,
Ann Arbor, MI, USA,
e-mail: jonxlee@umich.edu 
 }
 }
 
\abstracttcam{
Generalized inverses play a fundamental role in numerical linear algebra and statistical estimation, particularly when matrices are rectangular, singular, or rank deficient. Even when the underlying matrix is sparse, classical generalized inverses such as the Moore-Penrose pseudoinverse are typically dense, leading to high storage requirements, expensive matrix-vector multiplications, and reduced computational efficiency. Motivated by these limitations, we investigate the problem of computing sparse symmetric generalized inverses for sparse symmetric matrices, extending previous work on sparse generalized inverses for rectangular rank-deficient matrices.
We consider the problem of minimizing the vector 1-norm over the set of symmetric generalized inverses, using vector 1-norm minimization as a tractable surrogate for sparsity promotion. Our first contribution is a new equivalent characterization of symmetric generalized inverses of symmetric matrices, which yields a compact affine reformulation of the problem. Building on these formulations, we develop a Douglas-Rachford splitting (DRS) algorithm equipped with a closed-form projection onto the feasible affine space. The proposed method generates feasible sparse generalized inverses at every iteration and scales efficiently to large sparse instances.
Computational experiments compare the proposed DRS approach with exact linear-optimization formulations solved by a commercial optimizer, as well as with local-search heuristics from the literature. The results demonstrate that the proposed formulation and algorithm produce generalized inverses that are substantially sparser than the Moore-Penrose pseudoinverse while maintaining significantly smaller 1-norms than competing sparse approaches. Moreover, the DRS algorithm exhibits superior scalability relative to exact linear-optimization formulations, successfully solving instances far beyond the reach of commercial solvers.
As an application, we investigate the computation of least-squares solutions for problems involving many right-hand-side vectors and sparse rank-deficient design matrices. We establish theoretical connections between sparse generalized inverses of $A$ and sparse symmetric generalized inverses of $A^\top A$, leading to alternative closed-form least-squares solvers. Numerical experiments on both randomly generated sparse matrices and real benchmark instances from the SuiteSparse Matrix Collection arising from power-system applications illustrate the trade-offs among sparsity, numerical robustness, and computational cost. The results indicate that sparse symmetric generalized inverses can provide computationally efficient and numerically attractive alternatives for repeated least-squares computations in large-scale settings.
}


\acknowledgments{}

\acknowledgments{A. Machado was supported in part by CNPq grant 131219/2026-0.
M. Fampa was supported in part by CNPq grant 307167/2022-4.
J. Lee was supported in part by AFOSR grant FA9550-22-1-0172.} 

\counterwithout{equation}{section}

\begin{document}
	
\maketitle


\section{Introduction}\label{sec:introduction}

When a real matrix $A \in \mathbb{R}^{m \times n}$ is not square or square but not invertible, generalized inverses of $A$ are useful in several matrix algebra problems (see \cite{Rao1972GeneralizedIO}). In particular, the well-known Moore–Penrose (M-P) pseudoinverse 
can be used to compute solutions to several key problems in statistics, such as the least-squares solution of an overdetermined system of linear equations and the minimum 2-norm solution of an underdetermined system of linear equations.

In general, for numerical stability and computational efficiency, it is often desirable to compute generalized inverses that optimize, or approximately optimize, certain structural properties, such as minimum vector 1-norm, minimum rank, or maximum sparsity (the latter being NP-hard; see \cite[Proposition~1.3]{XFLPsiam}). Such generalized inverses will, in general, differ from the M-P pseudoinverse. In the setting considered in this work, we assume that the input matrix is symmetric and seek a symmetric generalized inverse with minimum 1-norm, using 1-norm minimization as a tractable mechanism for promoting sparsity while maintaining favorable numerical properties.

If $A=U\Sigma V^\top $ is the real singular value decomposition of $A$ (see \cite{golub13}, for example), then the M–P pseudoinverse of $A$ can be defined as $A^\dagger := V \Sigma^\dagger U^\top $, where $\Sigma^\dagger $ has the form of the transpose of the diagonal matrix $\Sigma$, and is derived from $\Sigma$ by taking the reciprocals of the non-zero (diagonal) elements of $\Sigma$ (i.e., the non-zero singular values of $A$).
We can define different tractable sparse generalized  inverses, based on the following fundamental characterization of the
M–P pseudoinverse.

\bigskip 

\begin{theorem}[\cite{Penrose}]
For $A\in\mathbb{R}^{m \times n}$, the M–P pseudoinverse $A^{\dagger}$ is the unique matrix 
$H\in\mathbb{R}^{n \times m}$ satisfying:
\begin{align}
& AHA = A \label{property1}\tag{P1}\\
& HAH = H \label{property2}\tag{P2}\\
& (AH)^{\top} = AH \label{property3}\tag{P3}\\
& (HA)^{\top} = HA. \label{property4}\tag{P4}
\end{align}
\end{theorem}

\noindent A \emph{generalized inverse} of $A$ is simply any matrix $H$ that satisfies \ref{property1}. 
If a generalized inverse satisfies additionally \ref{property2} (resp., \ref{property3}, \ref{property4}) we call it \emph{reflexive} (resp., \emph{ah-symmetric}, \emph{ha-symmetric}). A generalized inverse $H$ of $A$ always has rank at least the rank of $A$, and with equality if and only if 
$H$ is reflexive.

We now state a key structural result essential to our approach.

\begin{theorem}[\cite{BenIsrael1974}, p. 208, Example 14 (without~proof); see also  \cite{PFLX_ORL} (with proof)]\label{thm:structural}
For a matrix $A\in\mathbb{R}^{m\times n}$ of rank $r$, consider the complete singular value decomposition  $A =: U \Sigma V^\top$ with
\[
\Sigma =: \begin{bmatrix}\underset{\scriptscriptstyle r\times r}{D} & \underset{\scriptscriptstyle r\times (n-r)}{0}\\[1.5ex]
\underset{\scriptscriptstyle (m-r)\times r}{0} & \underset{\scriptscriptstyle (m-r)\times (n-r)}{0}\end{bmatrix},
\]

\noindent and $D$ diagonal.
For $H \in \mathbb{R}^{n \times m}$, let $\Gamma:=V^\top H U$ (hence $H= V\Gamma U^\top$), where we partition $\Gamma$ into blocks as in:
\[
\Gamma=:\begin{bmatrix}\underset{\scriptscriptstyle r\times r}{X} & \underset{\scriptscriptstyle r\times (m-r)}{Y}\\[1.5ex]
\underset{\scriptscriptstyle (n-r)\times r}{Z} & \underset{\scriptscriptstyle (n-r)\times (m-r)}{W}\end{bmatrix}.
\]
Then:
\begin{enumerate}
\item[(i)] {\rm\ref{property1}} is equivalent to $X = D^{-1}$.
\item[(ii)] If {\rm\ref{property1}} holds, then {\rm\ref{property2}} is equivalent to $ZDY = W$.
\item[(iii)] If {\rm\ref{property1}} holds, then {\rm\ref{property3}} is equivalent to $Y = 0$.
\item[(iv)] If {\rm\ref{property1}} holds, then {\rm\ref{property4}} is equivalent to $Z = 0$.
\end{enumerate}
\end{theorem}

In fact, setting $W$, $Y$, and $Z$ to null matrices in Theorem \ref{thm:structural}, we can see the simple relation between $A^\dagger$ and the SVD decomposition of $A$.

\medskip

\noindent{\bf  Literature overview.}
The idea of minimizing the vector 1-norm of matrix inverses and generalized inverses as a tractable approach for promoting sparsity has been investigated in several contexts in the recent literature. 
\cite{dokmanic,dokmanic1,dokmanic2} proposed the use of 1-norm minimization for inducing sparsity in left and right inverses of matrices. 
Building on these ideas, \cite{FFL2016} introduced a systematic framework for computing sparse generalized inverses satisfying different subsets of the M-P properties, formulating the associated problems as linear-optimization models based on vector 1-norm minimization.
Subsequent works investigated both theoretical and algorithmic aspects of sparse generalized inverses. 
\cite{FampaLee2018ORL} developed a combinatorial polynomial-time approximation algorithm for 1-norm minimization over reflexive generalized inverses. 
The work of \cite{XFLPsiam} extended the study to (i) symmetric generalized inverses of symmetric matrices and (ii) ah-symmetric generalized inverses of general matrices. That work also established hardness results for sparsity maximization problems associated with generalized inverses and derived sparsity bounds for solutions obtained through linear-optimization formulations. Related developments can be found in \cite{XFLrank12,FLPXjogo}.
Optimization formulations for sparse generalized inverses were further investigated in \cite{PFLX_ORL}, where structural properties of generalized inverses were exploited to derive compact formulations for both vector 1-norm minimization, aimed at sparse generalized inverses, and (2,1)-norm minimization, aimed at row-sparse generalized inverses. More recently, \cite{ponte2024goodfastrowsparseahsymmetric} developed  ADMM (Alternating Direction Method of Multipliers) algorithms for computing ah-symmetric reflexive generalized inverses using both 1-norm and (2,1)-norm minimization. That work also showed that the approximation algorithm of \cite{FampaLee2018ORL} provides approximation guarantees for the (2,1)-norm setting.
The connection between sparse generalized inverses and least-squares computation was recently explored in \cite{machado2025}, which investigated sparse \emph{universal solvers} for fundamental problems in statistics. That work studied sparse generalized inverses capable of simultaneously solving least-squares and minimum-norm problems, including minimum-rank universal least-squares solvers, and proposed several alternative linear-optimization formulations together with efficient proximal-point algorithms. In particular, the paper demonstrated the strong practical performance of Douglas-Rachford splitting (DRS) methods for computing sparse universal solvers.

The present paper continues this line of research by focusing on sparse symmetric generalized inverses of sparse symmetric matrices. In contrast to previous approaches based primarily on linear optimization or local-search techniques for computing sparse symmetric generalized inverses, motivated by the successful results in \cite{machado2025}, we develop  an efficient DRS algorithm with closed-form projections for this problem. We further investigate the application of these sparse generalized inverses to repeated least-squares computations involving sparse rank-deficient matrices.

A preliminary version of the present work 
was presented at ERMAC 2026. 
That work investigated sparse symmetric generalized inverses of symmetric matrices through 1-norm minimization and proposed a Douglas-Rachford splitting (DRS) framework for solving the associated optimization problem. The present paper substantially extends those preliminary results in several directions. In particular, we present a more extensive computational study, including experiments on significantly larger instances and detailed comparisons with local-search procedures from the literature. 
Moreover, we establish new applications of sparse symmetric generalized inverses to a generalized version of the Tikhonov regularization problem, as well as to sparse least-squares computations involving rank-deficient matrices. Finally, we investigate the practical performance of the proposed approaches through a case study based on real-world data, providing additional insight into the trade-offs between sparsity, numerical robustness, and computational efficiency in least-squares applications.

\medskip

\noindent{\bf Organization and Contributions.} 
This paper develops new theoretical and algorithmic tools for computing sparse symmetric generalized inverses of sparse symmetric matrices and investigates their application to least-squares problems.

Our first contribution is presented in Section~\ref{sec:sparse}, where we study the problem of computing a symmetric generalized inverse with minimum vector 1-norm. While the original formulation is naturally expressed through the M-P condition \ref{property1} together with symmetry constraints, we derive a new equivalent characterization for symmetric generalized inverses of symmetric matrices. Specifically, we prove that the conditions \ref{property1} and $H^\top =H$ are equivalent to a single matrix equation involving both $A$ and $H$. This reformulation leads to a more compact affine description of the feasible set.

In Section~\ref{sec:DRS}, we exploit our formulation to derive a new DRS  algorithm for sparse symmetric generalized inverses. A key contribution of this section is the derivation of an explicit analytical formula for the projection onto the affine set of symmetric generalized inverses. This enables each DRS iteration to be computed efficiently without solving auxiliary optimization subproblems. We also discuss initialization strategies, stopping criteria, and practical implementation details. The proposed method generates feasible generalized inverses throughout the iterative process and is specifically designed to scale to large sparse matrices.

Section~\ref{sec:numexp} contains a detailed computational study comparing three approaches:
(i) linear programming formulations solved with Gurobi,
(ii) the proposed DRS algorithm, and
(iii) previously developed local-search approaches from the literature.

The experiments evaluate sparsity, 1-norm, rank behavior, and computational time on randomly generated sparse low-rank matrices of dimensions up to $3500\times3500$. The results demonstrate several important phenomena:
\begin{itemize}
    \item minimizing the 1-norm effectively induces sparsity in generalized inverses;
    \item we can find  generalized inverses that are significantly sparser than the M-P pseudoinverse;
    \item there is a clear trade-off between sparsity and the minimum-rank property; 
    \item the proposed DRS method scales substantially better than exact linear-optimization formulations.
\end{itemize}

A second major contribution of the paper is developed in Section~\ref{sec:LS}, where we apply sparse generalized inverses to least-squares estimation problems. We establish theoretical links between:
\begin{itemize}
    \item generalized inverses of $A^\top A$, and
   ah-symmetric reflexive generalized inverses of $A$.
\end{itemize}

These results lead to a closed-form solution for generalized Tikhonov regularization problems and alternative closed-form solutions for least-squares problems. We then investigate two competing strategies for repeated least-squares computations:
\[
\hat{\theta} = \hat{H}A^\top b
\quad\text{and}\quad
\hat{\theta} = Hb,
\]
where $\hat{H}$ and $H$ are sparse generalized inverses obtained through different optimization models.

We further employ an efficient reformulation for computing sparse ah-symmetric reflexive generalized inverses and compare DRS-based methods with local-search procedures from the literature. In Section \ref{sec:exp-lse}, extensive numerical experiments on randomly generated sparse matrices evaluate the resulting least-squares solvers in terms of arithmetic complexity, sparsity, numerical scaling, and computational efficiency.

Finally, in Section~\ref{sec:case}, we present a case study using real sparse rank-deficient least-squares instances from the SuiteSparse Matrix Collection arising from applications at the New York Power Authority (NYPA). These experiments provide insight into the practical performance of the proposed methods on challenging real-world matrices and highlight the trade-offs between scalability, sparsity, and numerical robustness in practical least-squares computations.

\medskip

\noindent{\bf Notation.}
We denote the $i$-th standard unit vector by $e_i$ and the identity matrix of order $n$ by $I_n$\,. For $p\in\{0\}\cup [1,\infty]$, we denote the ordinary vector norm $p$ by $\|\cdot\|_p$ (it is evident that this is only a pseudonorm for $p=0$). The vectorization of a matrix $A\in\mathbb{R}^{m\times n}$ is denoted by $\vvec{(A)}\in\mathbb{R}^{mn}$, and its vector norm $p$ is defined by $\|A\|_p := \|\vvec(A)\|_p$\,. For matrices, $\rank(\cdot)$ denotes rank,  $\|\cdot\|_F$ denotes Frobenius norm, $\mathcal{R}(\cdot)$ denotes the range (i.e., column space),
and $\mathcal{K}(\cdot)$ denotes the kernel (i.e., null space). We denote the Kronecker product by $\otimes$ (using the identity: $\vvec(ABC)=(C^\top \otimes A)\vvec(B)$). We denote by $\argmin\{\cdot\}$ any solution to the associated minimization problem.

\section{Sparse symmetric generalized inverse of a symmetric matrix}\label{sec:sparse}

We consider the case where $A\in\mathbb{R}^{n\times n}$ is symmetric, and we seek a sparse symmetric generalized inverse of $A$. We observe that if $A$ is symmetric, then the M-P pseudoinverse of $A$ is symmetric; therefore, we know that symmetric generalized inverses of $A$ exist. We are interested in the following problem, which seeks to induce sparsity through the minimization of the vector norm 1.

\begin{align}
& \min_{H \in \mathbb{R}^{n \times n}}\{\|H\|_1 : {\rm \ref{property1}},\, H^\top \!= H\} \tag{$P_{1,\rm{Sym}}^1$}.\label{prob:p1_sym}
\end{align}

An alternative would be to simply seek a generalized inverse that minimizes the 1-norm  and then symmetrize it, which would necessarily still minimize the 1-norm.
However, with symmetrization, the rank
could increase substantially and
the sparsity could decrease substantially.

We can arrive at a subtle and useful reformulation of \ref{prob:p1_sym}
with the following characterization.

\begin{theorem}\label{thm:p1sym_characteristic}
Let $A\in\mathbb{R}^{n\times n}$ be symmetric. Then $H$ satisfies {\rm\ref{property1}} and $H^\top = H$ if and only if $H$ satisfies $AHA + H = A + H^\top$.
\end{theorem}

\begin{proof}
Let $A = V\Sigma V^\top$ be the full singular value decomposition of $A$. Consider $H = V\Gamma V^\top$ using the notation of Theorem \ref{thm:structural}.

$(\Rightarrow)$ It is easy to see that if $H$ satisfies \ref{property1} and $H^\top = H$, then $H$ satisfies $AHA + H = A + H^\top$.

$(\Leftarrow)$ To see that $H$ satisfies \ref{property1} and $H^\top = H$ if $H$ satisfies $AHA + H = A + H^\top$, observe that
    \begin{equation*}
        \begin{aligned}
            AHA + H & = A + H^\top & \Leftrightarrow \\
            V\Sigma V^\top V\Gamma V^\top V\Sigma V^\top + V\Gamma V^\top & = V\Sigma V^\top + V\Gamma^\top V^\top & \Leftrightarrow \\
            \Sigma \Gamma \Sigma + \Gamma & = \Sigma + \Gamma^\top & \Leftrightarrow \\
            \begin{bmatrix}
                D & 0 \\ 0 & 0
            \end{bmatrix}\begin{bmatrix}
                X & Y \\ Z & W
            \end{bmatrix}\begin{bmatrix}
                D & 0 \\ 0 & 0
            \end{bmatrix} + \begin{bmatrix}
                X & Y \\ Z & W
            \end{bmatrix} & = \begin{bmatrix}
                D & 0 \\ 0 & 0
            \end{bmatrix} + \begin{bmatrix}
                X^\top & Z^\top \\ Y^\top & W^\top
            \end{bmatrix} & \Leftrightarrow \\
            \begin{bmatrix}
                DXD + X & Y \\ Z & W
            \end{bmatrix} & = \begin{bmatrix}
                D + X^\top  & Z^\top  \\ Y^\top & W^\top
            \end{bmatrix}.
        \end{aligned}
    \end{equation*}
    Therefore, $Z = Y^\top$, $W = W^\top$, and $DXD = D + X^\top$. From $DXD = D + X^\top$, we have
    \begin{equation*}
        \begin{cases}
            \begin{aligned}
                d_i^2 x_{ii} + x_{ii} &= d_i + x_{ii}\,, & \text{if } i = j; \\
                d_i d_j x_{ij} + x_{ij} &= x_{ji}\,, & \text{if } i \neq j.
            \end{aligned}
        \end{cases}
    \end{equation*}
    Then, we have that   $d_i^2x_{ii} + x_{ii} = d_i + x_{ii} \Leftrightarrow x_{ii} = \textstyle\frac{1}{d_i}$. Adding the equations   $d_id_jx_{ij} + x_{ij} = x_{ji}$ e $d_id_jx_{ji} + x_{ji} = x_{ij}$\,, for $i\neq j$, we have that $d_id_jx_{ij} + x_{ij} + d_id_jx_{ji} + x_{ji} = x_{ji} + x_{ij} \Leftrightarrow x_{ij} = - x_{ji}$\,. Therefore, $d_id_jx_{ij} + x_{ij} = x_{ji} \Leftrightarrow d_id_jx_{ij} + x_{ij} = -x_{ij} \Leftrightarrow (d_id_j + 2)x_{ij} = 0$. Because $d_i, d_j > 0$,  we have that $x_{ij} = 0$, $\forall i \neq j$.
Thus, $X = D^{-1}$. From  Theorem \ref{thm:structural}, $H$ satisfies \ref{property1}. Because $Z = Y^\top$ and $W = W^\top$, $H$ also satisfies $H^\top = H$.
\end{proof}

Based on Theorem \ref{thm:p1sym_characteristic}, we present the following equivalent formulation for \ref{prob:p1_sym}:
\begin{align}
& \min_{H \in \mathbb{R}^{n \times n}}\left\{\|H\|_1 ~:~ AHA + H = A + H^\top\right\}. \tag{$\tilde{P}_{1,\rm{Sym}}^1$}\label{prob:p1_sym2}
\end{align}

\section{A DRS algorithm for \texorpdfstring{\ref{prob:p1_sym}}{A DRS algorithm for P11sym}}\label{sec:DRS}

We developed a Douglas-Rachford Splitting (DRS) algorithm to solve \ref{prob:p1_sym}\,.  DRS is a proximal method that has been widely applied in the literature to this type of problem with promising results (see, for example, \cite{dossal2024optimizationorderalgorithms, Fu_2020}).
It solves convex optimization problems of the form:

\begin{equation}\label{prob:unconstrained}
\text{min } f(H) + g(H),
\end{equation}
where $f, g : \mathbb{R}^{n \times m} \rightarrow \mathbb{R} \cup \{ \infty \}$ are convex, closed, and proper, by  iteratively calculating (see \cite{dossal2024optimizationorderalgorithms}, \cite{BauschkeCombettes2017})
\[
    \begin{array}{ll}
            &H^{k+1/2} := \prox_{\lambda f}(V^k), \\
            &V^{k+1/2} := 2H^{k+1/2} - V^k, \\
            &H^{k+1} := \prox_{\lambda g}(V^{k+1/2}),\\
            &V^{k+1} := V^{k} + H^{k+1} - H^{k+1/2},
    \end{array}
    \]
for $k=0,1,\ldots$, where $\prox_{\lambda h}$ is given by
\begin{equation*}
    \prox_{\lambda h}(V) := \argmin_H\left\{h(H) + \frac{1}{2\lambda}\|H - V\|_{F}^{2}\right\}.
\end{equation*}
We are interested in applying DRS to solve \ref{prob:p1_sym}\,, or, more generally, to solve a problem formulated as
\begin{equation}\label{general}
    \min \{\|H\|_1 : H \in \mathcal{C} \}, 
\end{equation}
where $\mathcal{C}$ is a closed convex set. \eqref{general} 
can be formulated as  \eqref{prob:unconstrained} taking  $f = \|\cdot\|_1$ and $g = \mathcal{I}_{\mathcal{C}}$ (the characteristic function of $\mathcal{C}$). In this case, we have from  \cite{Parikh_Boyd} that the respective proximal operators are given by  $\prox_{\lambda f}(X) = S_{\lambda}(X)$ and $\prox_{\lambda g}(X) = \Pi_{\mathcal{C}}(X)$,
where $S_{\lambda}(X)$ is the element-wise  \emph{soft thresholding} operator, defined by
\begin{equation*}
    S_{\lambda}(x) := \begin{cases}
        x - \lambda, \quad & a > \lambda;\\ 
        0, & |x| \leq \lambda;\\ 
        x + \lambda, & a < - \lambda,
    \end{cases}
\end{equation*}
and $\Pi_{\mathcal{C}}$ is  the projection onto $\mathcal{C}$.

Next, we discuss in more detail the application of DRS to problem \ref{prob:p1_sym}\,, starting with the projection onto the specific set $\mathcal{C}$ for the problem. In many cases, such as ours, where $\mathcal{C}$ is affine, the projection expression has an analytical solution, which we develop below. 

\subsection{\texorpdfstring{Projection onto $\mathcal{C}$}{Projection onto C}}\label{sec:proj_p1sym}

The set of all  symmetric generalized inverses of a given symmetric matrix $A$ can be expressed as the affine space $\mathcal{C} := \{H ~:~ AHA = A,~ H^\top = H \}$. We will now address the calculation of  projection onto this set.

\begin{proposition}
    If $\mathcal{C} := \{H ~:~ AHA = A,~ H^\top = H\}$, then 
    \[
    \Pi_{\mathcal{C}}(V) =\textstyle\frac{1}{2}(V + V^\top) -\textstyle\frac{1}{2}AA^\dagger(V + V^\top)A^\dagger A + A^\dagger.
    \]
\end{proposition}

\begin{proof}

Let $V \in \mathbb{R}^{n \times m}$ be a given matrix. We want to find a closed-form solution of
\begin{equation}\label{prob:proj_p1sym}
\argmin_H\{\|H-V\|_F^2 ~:~ AHA = A,~ H^\top = H \} .\tag{$\text{Proj}_{1,\rm{Sym}}$}
\end{equation}
The Lagrangian function for \ref{prob:proj_p1sym} is given by 
  \begin{equation}\label{eq:projp1sym_lagrangian}
        \mathcal{L}(H, \Lambda, \Gamma) = \|H-V\|_F^2 + \langle \Lambda, AHA - A \rangle + \langle \Gamma, H - H^\top\rangle.
    \end{equation}
    Because \ref{prob:proj_p1sym} is convex, its primal and dual solutions $(H, \Lambda, \Gamma)$ are such that $\nabla_H \mathcal{L}(H, \Lambda, \Gamma) = 0$, or equivalently, such that
    \begin{align}
            H & = V -\textstyle\frac{1}{2}(A\Lambda A + \Gamma - \Gamma^\top) & \Rightarrow \label{firsteq}\\
            AHA=A & = AVA -\textstyle\frac{1}{2}A(A\Lambda A + \Gamma - \Gamma^\top)A & \Leftrightarrow \nonumber\\
            A^2\Lambda A^2 & = 2(AVA - A) - A(\Gamma - \Gamma^\top)A & \Rightarrow \nonumber\\
            {A^\dagger}^2 A^2\Lambda A^2 {A^\dagger}^2 & = 2{A^\dagger}^2(AVA - A){A^\dagger}^2 - {A^\dagger}^2A(\Gamma - \Gamma^\top)A{A^\dagger}^2 & \Leftrightarrow \nonumber\\
            A^\dagger A\Lambda AA^\dagger & = 2(A^\dagger VA^\dagger - {A^\dagger}^3) - A^\dagger(\Gamma - \Gamma^\top)A^\dagger & \Leftrightarrow \nonumber\\
            ((AA^\dagger) \otimes (A^\dagger A))\vvec(\Lambda) & = \vvec(2(A^\dagger VA^\dagger - {A^\dagger}^3) - A^\dagger(\Gamma - \Gamma^\top)A^\dagger).\nonumber
        \end{align}
        
Because $(AA^\dagger) \otimes (A^\dagger A)$ is an orthogonal projection, $\Lambda = 2(A^\dagger VA^\dagger - {A^\dagger}^3) - A^\dagger(\Gamma - \Gamma^\top)A^\dagger$ is a solution to the last equation above. Substituting this value of $\Lambda$ into \eqref{firsteq}, we have

 \begin{equation*}
        \begin{aligned}
        H    & = V - A(A^\dagger VA^\dagger - {A^\dagger}^3)A + \textstyle\frac{1}{2}AA^\dagger(\Gamma - \Gamma^\top)A^\dagger A - \textstyle\frac{1}{2}(\Gamma - \Gamma^\top) & \Leftrightarrow \\
         H   & = V - AA^\dagger VA^\dagger A + A^\dagger + \textstyle\frac{1}{2}AA^\dagger(\Gamma - \Gamma^\top)A^\dagger A - \textstyle\frac{1}{2}(\Gamma - \Gamma^\top) & \Leftrightarrow \\
            H^\top & = V^\top - AA^\dagger V^\top A^\dagger A + A^\dagger + \textstyle\frac{1}{2}AA^\dagger(\Gamma^\top - \Gamma)A^\dagger A - \textstyle\frac{1}{2}(\Gamma^\top - \Gamma) & \Rightarrow \\
            H - H^\top & = V - V^\top - AA^\dagger(V - V^\top)A^\dagger A + AA^\dagger(\Gamma - \Gamma^\top)A^\dagger A- (\Gamma - \Gamma^\top) & \Rightarrow \\
            0 & = V - V^\top - AA^\dagger(V - V^\top)A^\dagger A + AA^\dagger(\Gamma - \Gamma^\top)A^\dagger A- (\Gamma - \Gamma^\top) & \Leftrightarrow \\
            (\Gamma - \Gamma^\top) - AA^\dagger(\Gamma - \Gamma^\top)A^\dagger A & = V - V^\top - AA^\dagger(V - V^\top)A^\dagger A.
        \end{aligned}
    \end{equation*}

Thus, $\Gamma = V$ solves the last equation above. Now, substituting the values of $\Lambda$ and $\Gamma$ into \eqref{firsteq}, we have:

\begin{equation*}
\begin{aligned}
H & = V - AA^\dagger V A^\dagger A + A^\dagger + \textstyle\frac{1}{2}AA^\dagger(V - V^\top)A^\dagger A - \textstyle\frac{1}{2}(V - V^\top) & \\
& = \textstyle\frac{1}{2}(V + V^\top) - \textstyle\frac{1}{2}AA^\dagger(V + V^\top)A^\dagger A + A^\dagger,
\end{aligned}
\end{equation*}
which is the expression for the projection.
\end{proof}

\subsection{Starting point and stopping criterion}\label{sec:stop_crit}

We consider $V^0:=A^\dagger$ as the starting point, which is a feasible solution for our problem and requires no additional computational cost, as it is already used in the calculation of the projection $\Pi_{\mathcal{C}}(\cdot)$.

For the stopping criterion, we assume that $\epsilon^{\rm abs}$ and $\epsilon^{\rm rel}$ are given positive tolerances and calculate $\tau^{k} := H^{k+1} - H^{k+1/2}=V^{k+1} - V^{k}$ in each iteration $k$. We interrupt the algorithm if $||\tau^{k}||_F \leq \epsilon^{\rm abs} + \epsilon^{\rm rel}||V^1-V^0||_F$\,, $k>0$, which searches for both a small difference between the last two iterations and a small difference relative to the difference between the first two iterations. Note that, from \cite{BauschkeCombettes2017}, we have that if the sequence $V^k$ ($k=0,1,2,\ldots$) converges to a fixed point of $F(V):= V + \Pi_{\mathcal{C}}( 2 S_{\lambda}(V) - V) - S_{\lambda}(V)$, then both sequences $H^{k+1/2}$ ($k=0,1,2,\ldots$) and $H^{k+1}$ ($k=0,1,2,\ldots$) converge to $H^*$, a solution to the problem \eqref{prob:unconstrained}.
Finally, we observe that we always obtain a feasible primal solution from the DRS algorithm, taking as output the matrix $H^{k+1}$ calculated in the last iteration $k$, because $H^{k+1}$ is obtained from a projection onto the feasible set of the problem addressed.


\section{Computation of symmetric sparse generalized inverse}\label{sec:numexp}

We note that \ref{prob:p1_sym} can easily be reformulated as a linear program. Characterizing the extreme points of its feasible region is of particular interest, as such a characterization yields an upper bound on the number of non-zero elements in the basic-feasible solutions of this linear-optimization problem and, consequently, an upper bound on the number of non-zero elements in the sparsest symmetric generalized inverse. We emphasize, however, that an arbitrary (sparse) extreme point may still have a large vector 1-norm, which is undesirable from a numerical perspective. For comparison purposes in our numerical experiments, we next present a result from \cite{XFLPsiam} characterizing these extreme points.

\begin{proposition}\label{prop:extpts_p1sym}{\rm\cite[Proposition 2.1]{XFLPsiam}} Suppose that $A\in\mathbb{R}^{n\times n}$ is symmetric and has rank $r$. Then, the extreme solutions of the linear programming problem associated with \ref{prob:p1_sym} have at most $r^2+r$ non-zero elements. Furthermore, this upper bound is reached for $n-2\geq r\geq 3$.
\end{proposition}

We present numerical results by applying Gurobi v12.0.1 for \ref{prob:p1_sym} and \ref{prob:p1_sym2}\,, and the DRS algorithm for \ref{prob:p1_sym}\,, described in Section~\ref{sec:DRS}. The experiments were implemented in Julia v1.11.3 using VSCodium v1.97.1 and executed on ``zebra'', a 16-core machine running Windows Server 2019 Standard, equipped with Intel Xeon E5-2667 v4 processors at 3.20 GHz (8 cores each) and 128 GB of RAM.
We randomly generated  75 dense matrices $B\in\mathbb{R}^{n\times n}$ of rank $r$ (five for each  $(n, r)$), using the MATLAB's \emph{sprand} function, following the procedure described in \cite{FLPXjogo}. For our test instances, we then calculated the symmetric matrices $A=B^\top B$.
We considered dimensions $n$ ranging from $100$ to $3500$, with $r=0.25n$. A time limit of $7200$ seconds was imposed in all experiments, and the reported statistics correspond to averages over the five instances associated with each configuration. The Gurobi parameters were set as follows:
\emph{Barrier convergence tolerance}: $10^{-5}$,
\emph{Feasibility tolerance}: $10^{-5}$ and
\emph{Optimality tolerance}: $10^{-5}$.
For DRS, we used $\epsilon^{\rm abs}= 10^{-5}$, $\epsilon^{\rm rel} = 10^{-3}$ and $\lambda = 10^{-2}$.

When applying Gurobi to \ref{prob:p1_sym} and \ref{prob:p1_sym2}\,, we were only able to solve instances with $n=100$ within the time limit for our initial test set ($n=100,200,\ldots,3500$). Consequently, we generated additional smaller instances with $n\in\{20,40,60,80\}$, again using five instances per dimension, in order to provide a more meaningful comparison between the two formulations. The average results are reported in Table~\ref{tab:gurobi}.
The first column reports the dimension $n$.

The second column shows the ratio between the 0-norm of the computed solution $H$ and the upper bound $r^2+r$ on the number of nonzero elements in extreme solutions of the linear programming reformulation of \ref{prob:p1_sym} (see Proposition~\ref{prop:extpts_p1sym}). Columns 3 and 4 compare the solution $H$ obtained with Gurobi against the M-P pseudoinverse $A^\dagger$, reporting the ratios $\|H\|_0/\|A^\dagger\|_0$ and $\|H\|_1/\|A^\dagger\|_1$, respectively. Column 5 reports the ratio between rank($H$) and $r$. 
The results highlight the effectiveness of 1-norm minimization in producing sparse symmetric generalized inverses: on average, the solutions contain fewer than 7\% of the nonzero entries of $A^\dagger$. Moreover, the 1-norm of the computed solutions is, on average, reduced by more than 50\% relative to $A^\dagger$. These gains in sparsity, however, come at the expense of increased rank, which more than doubles for most instances, illustrating the trade-off between sparsity and low rank.
We also observe that, as the matrix dimension increases, the number of nonzero entries in the computed solutions approaches the upper bound from Proposition~\ref{prop:extpts_p1sym}. This suggests that even a substantial improvement in the theoretical upper bound would likely reduce the number of nonzero entries by less than 10\%. Finally, although Gurobi performed slightly better on \ref{prob:p1_sym} for these smaller instances, neither formulation exhibited satisfactory scalability for this class of problems.

\begin{table}[H]
    \centering  
    \begin{tabular}{r|rrrr|r|r}      
        \multicolumn{1}{c|}{} & \multicolumn{4}{c|}{} & \multicolumn{1}{c|}{\ref{prob:p1_sym}} & \multicolumn{1}{c}{\ref{prob:p1_sym2}} \\[4pt]
        \hline
        \multicolumn{1}{c|}{$n$} & \multicolumn{1}{c}{$\frac{\|H\|_0}{r^2+r}$} & \multicolumn{1}{c}{$\frac{\|H\|_0}{\|A^\dagger\|_0}$\vphantom{$\frac{\|H\|_0^{H^H}}{\|A^\dagger\|_0^H}$}} & \multicolumn{1}{c}{$\frac{\|H\|_1}{\|A^\dagger\|_1}$} & \multicolumn{1}{c|}{$\frac{\textrm{r}(H)}{r}$} & \multicolumn{1}{c|}{$\substack{\rm{time}\\\rm{(seconds)}}$} & \multicolumn{1}{c}{$\substack{\rm{time}\\\rm{(seconds)}}$} \\[4pt]
        \hline
        20\vphantom{$\frac{H^H}{H^H}$} & 0.687 & 0.054 & 0.500 & 1.400 & 1.83 & 1.41 \\
        40 & 0.833 & 0.060 & 0.444 & 1.940 & 2.13 & 5.89 \\
        60 & 0.869 & 0.062 & 0.456 & 2.053 & 17.99 & 36.27 \\
        80 & 0.858 & 0.059 & 0.448 & 2.390 & 103.79 & 145.49 \\
        100 & 0.913 & 0.064 & 0.447 & 2.456 & 403.15 & 560.85
    \end{tabular}
    \caption{Results for Gurobi for \ref{prob:p1_sym} and \ref{prob:p1_sym2} ($r=0.25n$)}
    \label{tab:gurobi}
\end{table}

Similar statistics to those in Table \ref{tab:gurobi}  are presented in Table \ref{tab:results_drs_p1sym} for the solutions obtained with the DRS algorithm. Comparing the DRS results for $n\!=\!100$ with those  obtained by Gurobi, we observe that both methods produce solutions with essentially the same 1-norms, while the DRS solutions are less sparse, exhibiting 0-norms approximately 28\% greater on average. Nevertheless, the computational effort required by Gurobi renders it uncompetitive when compared with DRS. We also note that all solutions produced by DRS are substantially sparser than $A^\dagger$, once again demonstrating the effectiveness of 1-norm minimization for inducing sparsity in our models. In particular, DRS was able to solve all instances with dimensions up to $n=3500$ within the prescribed time limit, whereas Gurobi was only able to solve the smallest instances with $n=100$. Moreover, for the smaller instances, DRS required computational times that were roughly one order of magnitude smaller than those of Gurobi. Finally, we observe that the 0-norms of the solutions obtained by both Gurobi and DRS are, on average, close to $r^2+r$, the upper bound on the number of nonzero elements in extreme solutions of the associated linear programming formulations. Although DRS does not explicitly exploit extreme-point solutions, the resulting 0-norms remain consistently close to this bound, on average within 30\% above it, and even below it for the larger instances. As in the Gurobi experiments, the sparse solutions produced by DRS generally have ranks greater than $r$, the rank of $A^\dagger$, reinforcing the trade-off between increased sparsity and low rank in our approach.

Overall, these results indicate that the proposed sparsity-inducing model for symmetric generalized inverses is effective, and that the DRS algorithm provides a scalable and computationally efficient approach for solving the problem.

\begin{table}[H]
    \centering  
    \begin{tabular}{r|rrrrr}         
        \multicolumn{1}{c|}{$n$} & \multicolumn{1}{c}{$\frac{\|H\|_0}{r^2+r}$} & \multicolumn{1}{c}{$\frac{\|H\|_0}{\|A^\dagger\|_0}$} & \multicolumn{1}{c}{$\frac{\|H\|_1}{\|A^\dagger\|_1}$} & \multicolumn{1}{c}{$\frac{\textrm{r}(H)}{r}$} & \multicolumn{1}{c}{$\substack{\rm{time}\\\rm{(seconds)}}$} \\[4pt]
        \hline
 100 \vphantom{$\frac{H^0}{H^0}$}& 1.167 & 0.082 & 0.447 & 2.62 & 10.90  \\
        200 & 1.266 & 0.085 & 0.384 & 2.98 & 16.62  \\ 
        300 & 1.184 & 0.081 & 0.383 & 2.91 & 17.31  \\ 
        400 & 1.114 & 0.076 & 0.374 & 3.10 & 60.13  \\ 
        500 & 1.148 & 0.078 & 0.359 & 3.21 & 109.14  \\ 
        1000 & 1.048 & 0.073 & 0.331 & 3.36 & 411.77  \\ 
        1500 & 1.009 & 0.071 & 0.323 & 3.50 & 841.76  \\ 
        2000 & 0.988 & 0.070 & 0.315 & 3.53 & 1489.70  \\ 
        2500 & 0.982 & 0.070 & 0.310 & 3.59 & 2639.73  \\ 
        3000 & 0.967 & 0.070 & 0.303 & 3.60 & 3574.90  \\ 
        3500 & 0.962 & 0.070 & 0.297 & 3.66 & 4780.89  
    \end{tabular}
    \caption{Results for DRS for \ref{prob:p1_sym} ($r=0.25n$)}
    \label{tab:results_drs_p1sym}
\end{table}


\section{An application to least squares}\label{sec:LS}

Given $\lambda \geq 0$, $A\in\mathbb{R}^{m\times n}$, $b\in\mathbb{R}^{m}$, and $L\in\mathbb{R}^{p\times n}$, we consider the following  well-known generalization of the Tikhonov regularization problem
\begin{align}
    \min_{\theta\in\mathbb{R}^n} \left\{
    \|A\theta  -b\|_2^2
    + 
    \lambda \|L\theta \|_2^2
    \right\}.
    \label{probgenL} \tag{G$_L$}
\end{align}
We would like to mention that for $L:=I_n$\,, this is commonly known as 
\emph{ridge regression} or \emph{regression with $\ell_2$-regularization}. 
And for $L:=I_n$\,, but replacing  the squared 2-norm with the 1-norm,
we obtain instead what is commonly known as 
\emph{lasso regression} or \emph{regression with $\ell_1$-regularization}. 

From convexity of the objective function of \ref{probgenL}\,, we have that $\hat\theta$ solves \ref{probgenL} if and only if it satisfies

\begin{equation}\label{kktgl}
    A^\top A\hat\theta + \lambda L^\top L\hat\theta -  A^\top b = 0 \Leftrightarrow (A^\top A  + \lambda L^\top L) \hat\theta =  A^\top b.
\end{equation}

\begin{proposition}\label{thm:solvable}
    The linear system \eqref{kktgl} is solvable for every  $\lambda \geq 0$, $A\in\mathbb{R}^{m\times n}$, $b\in\mathbb{R}^{m}$, and $L\in\mathbb{R}^{p\times n}$.
\end{proposition}

\begin{proof}
    First, we consider $\lambda=0$, and we show that $A^\top b \in \mathcal{R}(A^\top A)$.
    Let $\tilde{b}$ be the projection of $b$ onto $\mathcal{R}(A)$. Notice that $b - \tilde{b}$ is orthogonal to every vector in $\mathcal{R}(A)$, in particular, $b - \tilde{b}$ is orthogonal to every column of $A$. So we have

    \begin{equation*}
        A^\top (b - \tilde{b}) = 0.
    \end{equation*}

 Because $\tilde{b} \in \mathcal{R}(A)$, we have that exists $\theta \in \mathbb{R}^n$ such that $\tilde{b} = A\theta $. Thus, $A^\top (b - \tilde{b}) = 0 \Leftrightarrow A^\top A\theta  = A^\top b$. Hence, we have that $A^\top b \in \mathcal{R}(A^\top A)$.

    Now, to extend the proof for $\lambda>0$, we show that $\mathcal{R}(A^\top A) \subset \mathcal{R}(A^\top A + \lambda L^\top L)$. We do it by showing that $\mathcal{K}(A^\top A + \lambda L^\top L) \subset \mathcal{K}(A^\top A)$.
    Let $y \in \mathcal{K}(A^\top A + \lambda L^\top L)$. We have $(A^\top A + \lambda L^\top L)y = 0 \Leftrightarrow A^\top Ay = -\lambda L^\top Ly \Rightarrow y^\top A^\top Ay = -\lambda y^\top L^\top Ly$. Because $y^\top A^\top Ay \geq 0$ and $-\lambda y^\top L^\top Ly \leq 0$, we have $y^\top A^\top Ay = 0 ~\Leftrightarrow~ \langle Ay, Ay \rangle= 0 \Leftrightarrow Ay = 0$. Thus, $y \in \mathcal{K}(A)$. Because $\mathcal{K}(A) \subset \mathcal{K}(A^\top A)$, we have that $y \in \mathcal{K}(A^\top A)$. Hence, $\mathcal{K}(A^\top A + \lambda L^\top L) \subset \mathcal{K}(A^\top A)$, and therefore, $\mathcal{R}(A^\top A) \subset \mathcal{R}(A^\top A + \lambda L^\top L)$.

    Therefore, because $A^\top b \in \mathcal{R}(A^\top A)$ and $\mathcal{R}(A^\top A) \subset \mathcal{R}(A^\top A + \lambda L^\top L)$, we have that $A^\top b \in \mathcal{R}(A^\top A + \lambda L^\top L)$. Which means that the linear system in \eqref{kktgl} is solvable for every $\lambda \geq 0$, $A \in \mathbb{R}^{m \times n}$, $b \in \mathbb{R}^{m}$ , and $L \in \mathbb{R}^{p \times n}$.
\end{proof}

Next, we establish a closed-form solution for \ref{probgenL}\,. 
\begin{theorem}\label{thm:solvegen}
     If $\hat{H}$ is a generalized inverse of $\hat A:= A^\top A  + \lambda L^\top L$, then 
     \begin{equation}\label{solutionLSE1}
     \hat\theta:=\hat{H} A^\top b
     \end{equation}
     solves {\rm\ref{probgenL}} for every $b\in\mathbb{R}^m$.
\end{theorem}

\begin{proof}
    Let $f(\theta) = \|A\theta  - b\|_2^2 + \lambda\|L\theta\|_2^2$, then the gradient of problem \ref{probgenL} is
    \begin{equation}\label{gradGL}
        \nabla_{\theta} f(\theta) = 2(A^\top A + \lambda L^\top L)\theta - 2A^\top b.
    \end{equation}

    We show that the gradient of $f$ evaluated at $\hat\theta = \hat{H}A^\top b$ vanishes. Plugging $\hat\theta = \hat{H}A^\top b$ in \eqref{gradGL} we have $\nabla_{\theta} f(\hat{H}A^\top b) = 2(A^\top A + \lambda L^\top L)\hat{H}A^\top b - 2A^\top b = 2\hat{A}\hat{H}A^\top b - 2A^\top b$. By Proposition \ref{thm:solvable}, $A^\top b$ is in the range of $\hat{A}$, thus, there exists $y \in \mathbb{R}^n$ such that $\hat{A}y = A^\top b$. Replacing $A^\top b$ with $\hat{A}y$ in the expression of the gradient we have, $\nabla_{\theta} f(\hat{H}A^\top b) = 2\hat{A}\hat{H}\hat{A}y - 2\hat{A}y \Rightarrow \nabla_{\theta} f(\hat{H}A^\top b) = 2\hat{A}y - 2\hat{A}y \Rightarrow \nabla_{\theta} f(\hat{H}A^\top b) = 0$, where we used that $\hat{H}$ satisfies \ref{property1} with respect to $\hat{A}$.
\end{proof}

We note that when $\lambda=0$, \ref{probgenL} reduces to the least-squares (estimation) problem 
\begin{align}
    \min_{\theta\in\mathbb{R}^n} \left\{
    \|A\theta  -b\|_2^2
    \right\}.
    \label{lse} \tag{LSE}
\end{align}
Therefore, Theorem \ref{thm:solvegen} establishes a closed-form solution for \ref{lse} by setting $\lambda=0$. The following well-known proposition establishes an alternative formula for computing a solution to \ref{lse}.  

\begin{proposition}[see e.g. \cite{RohdeThesis}]\label{prop:p1p3}
If $H$ is an ah-symmetric generalized inverse of $A$, then 
\begin{equation}\label{solutionLSE2}
\hat{\theta}:=Hb
\end{equation}
solves {\rm\ref{lse}} for all $b\in\mathbb{R}^m$.
\end{proposition}

Motivated by the discussion above,  we investigate the  connection between 
   $H$ being an ah-symmetric generalized inverse of $A$ and $\hat{H}$ being a generalized inverse of $A^\top A$. 

\begin{theorem}\label{thm:hhbar}
    There exists a generalized inverse $H$ of $A$ and a generalized inverse $\hat{H}$ of $A^\top A$ such that $H = \hat{H}A^\top $.
\end{theorem}

\begin{proof}
Consider the full singular-value decomposition\index{full singular-value decomposition} $A =: U \Sigma V^\top$ with 
\[
\Sigma =: \begin{bmatrix}\underset{\scriptscriptstyle r\times r}{D} & \underset{\scriptscriptstyle r\times (n-r)}{0}\\[1.5ex]
\underset{\scriptscriptstyle (m-r)\times r}{0} & \underset{\scriptscriptstyle (m-r)\times (n-r)}{0}\end{bmatrix},
\]
\noindent with $D$ diagonal. 

Let $H \in \mathbb{R}^{n \times m}$ and $\Gamma:=V^\top H U$, where

\[
\Gamma=: \begin{bmatrix}\underset{\scriptscriptstyle r\times r}{X} & \underset{\scriptscriptstyle r\times (m-r)}{Y}\\[1.5ex]
\underset{\scriptscriptstyle (n-r)\times r}{Z} & \underset{\scriptscriptstyle (n-r)\times (m-r)}{W}\end{bmatrix}.
\]
Then $H= V\Gamma U^\top$.

Notice that $A^\top A = V\Sigma^2V^\top $ is the full singular-value decomposition of $A^\top A$.

 Let $\hat{H}  \in \mathbb{R}^{m \times m}$ and  $\Delta:=V^\top \hat{H}V$, where 

    \begin{equation*}
        \Delta := \begin{bmatrix}
            \underset{\scriptscriptstyle r\times r}{\Delta_{11}} & \underset{\scriptscriptstyle r\times n-r}{\Delta_{12}} \\
            \underset{\scriptscriptstyle n-r\times r}{\Delta_{21}} & \underset{\scriptscriptstyle n-r\times n-r}{\Delta_{22}}
        \end{bmatrix}.
    \end{equation*}
    Then  $\hat{H}=V\Delta V^\top $.

    Applying Theorem \ref{thm:structural} to $A$ and $A^\top A$,      
    we have that $H$ is a generalized inverse of $A$ if and only if $X = D^{-1}$, and $\hat{H}$ is a generalized inverse of $A^\top A$ if and only if $\Delta_{11} = D^{-2}$. Then,
    \begin{equation*}
        \begin{aligned}
            H & = V\begin{bmatrix}D^{-1} & Y \\ Z & W\end{bmatrix}U^\top 
        \end{aligned}
    \end{equation*}
and
    \begin{equation*}
        \begin{aligned}
            \hat{H}A^\top  & = V\begin{bmatrix}D^{-2} & \Delta_{12} \\ \Delta_{21} & \Delta_{22}\end{bmatrix}VV^\top \begin{bmatrix}D & 0 \\ 0 & 0\end{bmatrix}U^\top  \\
            & = V\begin{bmatrix}D^{-1} & 0 \\ \Delta_{21}D & 0\end{bmatrix}U^\top\,.
        \end{aligned}
    \end{equation*}
    Thus, we have that  $H = \hat{H}A^\top  \Leftrightarrow Y = W = 0, Z = \Delta_{21}D$. 
\end{proof}

\begin{corollary}\label{cor:ah-ref}
    If the equation $H = \hat{H}A^\top $ holds with $H$ being a generalized inverse of $A$ and $\hat{H}$ being a generalized inverse of $A^\top A$, then $H$ is an ah-symmetric reflexive generalized inverse of $A$.
    \label{cor:geninv_ahrefgeninv_relation}
\end{corollary}

\begin{proof}
    The result follows from the proof of Proposition \ref{thm:hhbar}, by noting that $X = D^{-1}$, $Y = 0$, $W = 0$ and applying Theorem \ref{thm:structural}.
\end{proof}

Finally, from Theorem \ref{thm:solvegen} and Proposition \ref{prop:p1p3}, we have that 

\begin{itemize}
    \item
if $\hat{H}$ is a generalized inverse of $A^\top A$, then $\hat\theta :=\hat{H}A^\top b$ solves \ref{lse}  for all $b$; 
    \item
      if $H$ is an ah-symmetric generalized inverse of $A$, then $\hat\theta:=H b$ solves \ref{lse}  for all $b$. 
\end{itemize}


\section{Numerical experiments for least squares with synthetic matrices}\label{sec:exp-lse}

    We aim at solving \ref{lse}, by applying the two closed-form solutions discussed in Section \ref{sec:LS}. Specifically,  we consider the following two methods for solving \ref{lse}: \begin{itemize}
    \item[($i$)]   $\hat\theta\!=\!\hat{H}A^\top b$, where $\hat{H}$ is a generalized inverse of $A^\top A$;
     \item[($ii$)]  $\hat\theta\!=\!Hb$, where  $H$ is an ah-symmetric reflexive generalized inverse of $A$.
     \end{itemize}
    
    For sparse matrices $A$ we seek sparse $\hat{H}$ and $H$, and compare the two methods that use them to compute a solution to \ref{lse}. 

    Because $\hat{A}:=A^\top A$ is symmetric, we will search for symmetric generalized inverses of $\hat{A}$ in method ($i$), solving the following problem.  
    \begin{align}
        &  
        \min_{\hat{H} \in \mathbb{R}^{n \times n}}\{\|\hat{H}\|_1 : \hat{A}\hat{H}\hat{A}=\hat{A},\, \hat{H}=\hat{H}^\top \}. \tag{$\hat{P}_{1,\mbox{\tiny{sym}}}^1$}\label{probHhat1sym}
\end{align}
    
    Moreover, as $\hat{H}A^\top$ is an ah-symmetric reflexive generalized inverse of $A$ (see Corollary \ref{cor:ah-ref}), to have a better comparison between the two methods, in method ($ii$), we impose that $H$ is not only an ah-symmetric generalized inverse of $A$, sufficient requirement to solve \ref{lse}, but it is also reflexive. We recall that the reflexive property gives the interesting property to the ah-symmetric generalized of having minimum  rank, as discussed in Section \ref{sec:introduction}.
    So, we consider the following problem:
    \begin{align*}
\tag{\mbox{$P^1_{123}$}}
\label{prob:min1norm}
\min \left\{\|H\|_{1}
~:~ \rm{\ref{property1}},\,\rm{\ref{property2}},\,\rm{\ref{property3}}
\right\}.
\end{align*}

From Theorem \ref{thm:structural}, we can deduce the following result. 
\begin{theorem}\label{thm:reduced}
    {\rm\ref{prob:min1norm}}  
can be efficiently reformulated  as 

\begin{equation}\tag{\mbox{$\mathcal{P}^1_{123}$}}\label{prob:barB}
\begin{array}{rrl}
&\min\limits_{\scriptscriptstyle{Z\in\mathbb{R}^{(n-r)\times r}}}& \left\| V \begin{bmatrix}
        D^{-1} & 0\\
        Z & 0\end{bmatrix}U^\top\right\|_1 \\[15pt]
        = &\min\limits_{\scriptscriptstyle{Z\in\mathbb{R}^{(n-r)\times r}}} &\left\| G + V_2ZU_1^\top\right\|_1\\
        [10pt]
        = &\min\limits_{\genfrac{}{}{0pt}{}{F\in\mathbb{R}^{n\times m},}{Z\in\mathbb{R}^{(n-r)\times r}}} &\sum_{i=1}^{n}\sum_{j=1}^m F_{ij}\\[5pt]
        &\mbox{\rm s.t.}& F -  V_2ZU_1^\top\geq G,\\[5pt]
       && F +  V_2ZU_1^\top\geq -G,
\end{array}
\end{equation}
where 
$A =: U \Sigma V^\top$ is the singular-value decomposition of $A$, where $U:= \begin{bmatrix}\underset{\scriptscriptstyle m\times r}{U_1} & \underset{\scriptscriptstyle m\times (m-r)}{U_2}\end{bmatrix} \in \mathbb{R}^{m \times m}$, $V:= \begin{bmatrix}\underset{\scriptscriptstyle n\times r}{V_1} & \underset{\scriptscriptstyle n\times (n-r)}{V_2}\end{bmatrix} \in \mathbb{R}^{n \times n}$ are orthogonal matrices  ($U^\top U = I_m, V^\top V= I_n$), and $\Sigma \in \mathbb{R}^{m \times n}$ with 
\[
\Sigma =: \begin{bmatrix}\underset{\scriptscriptstyle r\times r}{D} & \underset{\scriptscriptstyle r\times (n-r)}{0}\\[1.5ex]
\underset{\scriptscriptstyle (m-r)\times r}{0} & \underset{\scriptscriptstyle (m-r)\times (n-r)}{0}\end{bmatrix},
\]
with $D$ being a diagonal matrix with rank $r$, and where $ G:=V_1D^{-1}U_1^\top$\,.

\end{theorem}

\medskip 

With the solution $\hat{Z}$ of \ref{prob:barB}\,, we construct the ah-symmetric reflexive generalized inverse of $A$, the matrix $H:=  V\Gamma U^\top$, where (see Theorem \ref{thm:structural})
\[
\Gamma=: \begin{bmatrix}\underset{\scriptscriptstyle r\times r}{D^{-1}} & \underset{\scriptscriptstyle r\times (m-r)}{0}\\[1.5ex]
\underset{\scriptscriptstyle (n-r)\times r}{\hat{Z}} & \underset{\scriptscriptstyle (n-r)\times (m-r)}{0}\end{bmatrix}.
\]

Besides considering the matrix $H$ computed with the solution of \ref{prob:barB}\,
(solved with the DRS algorithm from \cite[Section 3]{machado2025}), 
and the solution $\hat{H}$ of 
\ref{probHhat1sym}\,, 
(solved with Gurobi), 
we also consider the matrix $H$ obtained by the local-search (LS) procedure for ah-symmetric reflexive generalized inverses applied to $A$ from 
\cite{XFLPsiam},
and  the matrix $\hat{H}$ obtained by the local-search
 procedure for the symmetric reflexive generalized inverses applied to $A^\top A$ from 
 \cite[Section 3]{PFLX_ORL}.
 The time limit to solve each instance with each method was set to $1200$ seconds.

 To compare the different methods for solving \ref{lse}, we  randomly generated 50 instances (five for each configuration given by $(m,n,r,d)$), with varied  ranks and  densities, with MATLAB's function \emph{sprand}. In order to obtain sparse low-rank matrices, we used sprand to generate an $m\times r$ matrix, and then added $n-r$ columns to it, given by linear combinations of the first $r$ columns.  
The parameters used were $m=1000$, $n=100,200,\ldots,500$, $r=0.75n$, and density $d=0.1, 0.25$. To compute the 0-norm, we considered the tolerance for non-zero elements as $10^{-5}$. 

The stopping criterion for  the DRS algorithm to solve \ref{prob:barB}  is similar to one presented in Section~\ref{sec:DRS}. The parameters used in Gurobi and the DRS algorithm  are:
\begin{itemize}
    \item Gurobi: barrier convergence tolerance, optimality and feasibility tolerances of $10^{-5}$;
    \item DRS: $\epsilon^{\text{abs}} := 10^{-5}$, $\epsilon^{\text{rel}} := 10^{-3}$, $\lambda := 10^{-2}$;
\end{itemize}

In Table~\ref{tab:lswithA}, we compare the performance of the DRS and LS methods for solving problem~\ref{prob:barB}. With respect to the objective of minimizing the 1-norm, DRS is considerably more effective, producing solutions whose 1-norms range from (80\%) to (95\%) of the 1-norm of the M-P pseudoinverse $A^\dagger$, whereas LS typically returns solutions with approximately twice the 1-norm of $A^\dagger$.

However, an important goal of 1-norm minimization in our framework is to induce sparsity, that is, to reduce the 0-norm of the solutions. From this perspective, LS outperforms DRS, producing solutions with fewer than 40\% of the nonzero entries of $A^\dagger$, while the DRS solutions retain approximately 95\% of $\|A^\dagger\|_0$\,.

Therefore, although the LS solutions are substantially sparser, they come at the cost of significantly larger 1-norms, which may lead to numerical instability or other undesirable effects in practical applications.

\begin{table}[H]
\centering   
        \begin{tabular}{r|r|r|r|r|r|r|r|r|r}
         \multicolumn{2}{c}{}  &\multicolumn{2}{c|}{}&  \multicolumn{3}{c|}{\ref{prob:barB}(DRS)} &  \multicolumn{3}{c}{\ref{prob:barB}(LS)} \\  [4pt]    
        \multicolumn{1}{c|}{$n,d$}  &\multicolumn{1}{c|}{$\|A\|_0$} &  \multicolumn{1}{c|}{$\|A^\dagger\|_0$} &  \multicolumn{1}{c|}{$\|A^\dagger\|_1$} &  \multicolumn{1}{c|}{$\|H\|_0$} &  \multicolumn{1}{c|}{$\|H\|_1$}&  \multicolumn{1}{c|}{Time}&  \multicolumn{1}{c|}{$\|H\|_0$} &  \multicolumn{1}{c|}{$\|H\|_1$}&  \multicolumn{1}{c}{Time}\\[4pt]
        \hline
        100\vphantom{$\frac{H^0}{H^0}$}, ~0.10 & 32137.2 & 99526.2 & 258.34 & 97288.0 & 206.5 & 5.85 & 36051.6 & 435.1 & 9.25 \\
        100, ~0.25 & 41574.0 & 99425.6 & 199.14 & 97496.6 & 185.0 & 5.04 & 36375.8 & 325.8 & 6.43 \\
        200, ~0.10 & 64247.4 & 199007.0 & 555.7 & 191919.0 & 471.1 & 9.74 & 73463.4 & 1131.4 & 6.59 \\
        200, ~0.25 & 83167.6 & 198648.4 & 419.7 & 191382.6 & 395.1 & 11.00 & 73629.0 & 791.9 & 6.63 \\
        300, ~0.10 & 96437.6 & 298446.2 & 891.6 & 283507.8 & 784.5 & 15.01 & 110780.8 & 1959.7 & 6.67 \\
        300, ~0.25 & 124745.0 & 297809.0 & 661.8 & 281664.8 & 629.4 & 16.97 & 111290.0 & 1384.3 & 6.61 \\
        400, ~0.10 & 128555.0 & 397871.4 & 1276.3 & 372008.2 & 1149.4 & 20.69 & 148455.8 & 3124.2 & 6.64 \\
        400, ~0.25 & 166394.4 & 396934.0 & 933.3 & 368272.8 & 889.3 & 24.45 & 148639.4 & 2094.5 & 6.54 \\
        500, ~0.10 & 160660.2 & 497214.6 & 1717.2 & 457951.2 & 1570.3 & 27.27 & 186202.8 & 4441.5 & 6.89 \\
        500, ~0.25 & 207940.4 & 496043.0 & 1239.8 & 451983.4 & 1180.5 & 31.94 & 186405.2 & 3036.3 & 7.02 \\
        \end{tabular}
    \caption{Computing  ah-symmetric reflexive  generalized inverse of $A\!\in\!\mathbb{R}^{1000\times n}$ with rank $r=0.75n$ and density~$d$ 
    }
    \label{tab:lswithA}
\end{table}

In Table~\ref{tab:lswithhatA}, we compare the performance of the methods for solving \ref{probHhat1sym}\,. The results exhibit trends similar to those observed in Table~\ref{tab:lswithA}, although the differences between the methods are more pronounced.

On average, the DRS solutions have 0-norms and 1-norms equal to 53\% and 59\%, respectively, of the corresponding values for $A^\dagger$. In contrast, the LS solutions attain an average 0-norm equal to only 30\% of that of $A^\dagger$, but at the expense of an average 1-norm equal to 220\% of the value for $A^\dagger$.

These results indicate that DRS is able to produce solutions that are simultaneously sparse and numerically well-scaled, achieving reductions in both the 0-norm and the 1-norm. Although LS yields even sparser solutions, the substantially larger 1-norms suggest that these solutions may be considerably more prone to numerical instability and error propagation in practical applications.

\begin{table}[H]
\centering    
        \begin{tabular}{r|r|r|r|r|r|r|r|r}
         \multicolumn{1}{c}{}  &\multicolumn{2}{c|}{} &  \multicolumn{3}{c|}{\ref{probHhat1sym}(DRS)} &  \multicolumn{3}{c}{\ref{probHhat1sym}(LS)} \\ [4pt]     
        \multicolumn{1}{c|}{$n,d$}   &  \multicolumn{1}{c|}{$\|\hat{A}^\dagger\|_0$} &  \multicolumn{1}{c|}{$\|\hat{A}^\dagger\|_1$} &  \multicolumn{1}{c|}{$\|\hat{H}\|_0$} &  \multicolumn{1}{c|}{$\|\hat{H}\|_1$}&  \multicolumn{1}{c|}{Time}&  \multicolumn{1}{c|}{$\|\hat{H}\|_0$} &  \multicolumn{1}{c|}{$\|\hat{H}\|_1$}&  \multicolumn{1}{c}{Time}\\[4pt]
        \hline
            100\vphantom{$\frac{H^0}{H^0}$}, ~0.10 & 9733.6 & 12.03 & 5161.6 & 6.3 & 2.92 & 2794.2 & 71.6 & 6.57 \\
            100, ~0.25 & 9362.8 & 5.7 & 4955.4 & 3.0 & 3.05 & 2815.8 & 39.5 & 6.74 \\
            200, ~0.10 & 38351.2 & 37.8 & 20338.6 & 20.7 & 7.87 & 11236.4 & 451.2 & 11.56 \\
            200, ~0.25 & 36544.4 & 18.0 & 19343.4 & 9.9 & 13.32 & 11388.0 & 275.2 & 8.12 \\
            300, ~0.10 & 85624.8 & 80.1 & 45486.8 & 47.2 & 14.25 & 25473.4 & 1896.9 & 14.53 \\
            300, ~0.25 & 80814.0 & 38.2 & 42819.6 & 22.7 & 22.27 & 25270.6 & 873.7 & 14.70 \\
            400, ~0.10 & 151248.8 & 146.0 & 80557.2 & 91.5 & 29.06 & 45201.6 & 4229.1 & 29.45 \\
            400, ~0.25 & 142344.8 & 69.0 & 75856.8 & 43.5 & 51.94 & 45142.4 & 2107.6 & 29.60 \\
            500, ~0.10 & 235509.6 & 245.8 & 125826.2 & 162.3 & 43.56 & 70441.0 & 10278.5 & 96.17 \\
            500, ~0.25 & 222386.0 & 115.4 & 118795.4 & 76.3 & 79.35 & 70225.4 & 3922.4 & 95.24 \\
        \end{tabular}
    \caption{Computation of symmetric generalized inverse of $\hat{A}:=A^{\protect\top}A\in\mathbb{S}^{n}$ with rank  $r=0.75n$ and density $d$
    }
    \label{tab:lswithhatA}
\end{table}

Comparing the computational times reported in Tables~\ref{tab:lswithA} and \ref{tab:lswithhatA}, we observe that, for the instances in Table~\ref{tab:lswithA}, DRS requires at most three times the computational time of LS, whereas in Table~\ref{tab:lswithhatA} both methods exhibit running times of the same order of magnitude. Since all instances are solved within 100 seconds, these results suggest that, for problems of this size, the choice between DRS and LS may depend primarily on the specific requirements of the application, as both methods provide solutions within reasonable computational times.

The main comparison between Tables~\ref{tab:lswithA} and \ref{tab:lswithhatA} arises in the context of solving least-squares problems by computing either $Hb$ or $\hat{H}A^\top b$. The computational cost of these approaches can be estimated by the number of scalar products required in the corresponding matrix-vector multiplications, namely $\|H\|_0$ and $\|\hat{H}\|_0+\|A^\top\|_0$\,, respectively.

The average values of $\|H\|_0$\,, using either DRS or LS solutions, are approximately 279000 and 111000, respectively. In contrast, the average values of $\|\hat{H}\|_0+\|A^\top\|_0$ are approximately 163000 and 140000, respectively, for DRS and LS. Therefore, from the perspective of minimizing the number of arithmetic operations, the best option is to use an ah-symmetric reflexive generalized inverse of $A$ obtained with LS.

However, this advantage comes with a potential drawback: the LS solutions have, on average, 1-norms approximately twice as large as those of $A^\dagger$, which may lead to numerical instability in practical applications. If numerical robustness is a priority, then the preferable alternative is to use a symmetric generalized inverse of $A^\top A$ computed with DRS. Although this approach requires approximately 46\% more matrix products than the LS solution, it produces solutions with 1-norms smaller than those of $A^\dagger$, thereby mitigating potential numerical difficulties.

Finally, we note an important limitation of this comparison between Tables~\ref{tab:lswithA} and \ref{tab:lswithhatA}: the matrices $A$ considered in the experiments are relatively sparse, with densities of either 10\% or 25\%. 

Consequently, for applications involving even sparser matrices, the approach based on computing $\hat{H}A^\top b$
 may become even more advantageous. On the other hand, for applications with dense matrices, computing $Hb$ is likely to provide better performance.

\section{Case study}\label{sec:case}

In the following numerical experiment, we revisit the comparison between the two approaches for computing least-squares solutions discussed in the previous sections, now using real sparse rank-deficient least-squares instances generated at the New York Power Authority (NYPA). These matrices, originally contributed by Deepak Maragal, are available in the SuiteSparse Matrix Collection~\cite{davis2011university} and arise from practical sparse rectangular systems appearing in power-system analysis and engineering estimation problems.

The objective of this case study is twofold. First, we aim to evaluate whether the conclusions obtained from randomly generated instances remain valid in practical applications involving highly structured sparse matrices. Second, we seek to better understand the trade-offs between sparsity, numerical robustness, and computational scalability when sparse generalized inverses are applied to real least-squares problems.

The results are reported in Tables~\ref{tab:lswithA_cs} and \ref{tab:lswithhatA_cs}. As in Section~\ref{sec:exp-lse}, Table~\ref{tab:lswithA_cs}  presents the results associated with computing ah-symmetric reflexive generalized inverses of $A$, while Table~\ref{tab:lswithhatA_cs} reports the results for symmetric generalized inverses of $\hat{A}:=A^\top A$. A time limit of 24 hours was imposed for every experiment in this section.

We first analyze the results of Table~~\ref{tab:lswithA_cs}. For the instances successfully solved by DRS within the prescribed time limit, the solutions obtained by DRS and LS exhibit comparable sparsity levels, as measured by the 0-norm. In particular, for the instances with dimensions $n=350$ and $n=860$, both methods produce generalized inverses with numbers of nonzero entries of the same order of magnitude. This behavior contrasts with the synthetic instances from Section~\ref{sec:exp-lse}, where LS consistently produced substantially sparser solutions than DRS.

A second important observation concerns the 1-norm of the computed generalized inverses. In the synthetic experiments, LS solutions typically exhibited 1-norms significantly larger than those of the M-P pseudoinverse. In the NYPA instances, however, this phenomenon is less pronounced. In several cases, the increase in the 1-norm associated with LS is relatively moderate, and in one particularly remarkable instance ($n=1034$), the LS solution attains a significantly smaller 1-norm than the M-P pseudoinverse. More precisely, while the M-P pseudoinverse has 1-norm of order $10^8$, the LS solution has 1-norm of order $10^4$. This suggests that the structural properties of the NYPA matrices differ  from those of the randomly generated test instances and may favor sparse generalized inverses with improved numerical scaling.

The computational-time comparison in Table~\ref{tab:lswithA_cs}  highlights the practical advantages of the LS approach for very large matrices. While DRS requires from several minutes to several hours for the larger instances, LS typically solves the same problems in fractions of a second or a few seconds. In particular, LS solves the instances with $n=350$, $860$, and $1034$ in approximately $0.1$, $1.4$, and $2$ seconds, respectively, whereas DRS requires approximately $143$, $1476$, and more than $24$ hours for the corresponding instances. These results indicate that LS is more scalable than DRS on this class of real-world matrices.

The results reported in Table~\ref{tab:lswithhatA_cs} exhibit trends similar to those observed in Table~\ref{tab:lswithA_cs}, but now for the computation of symmetric generalized inverses of $\hat{A}=A^\top A$. Once again, DRS consistently produces solutions with significantly smaller 1-norms than LS, reinforcing the ability of the proposed optimization-based framework to generate numerically well-scaled sparse generalized inverses. However, this advantage comes at a substantial computational cost. While LS remains capable of solving instances up to dimension $n=1034$, DRS only solves instances up to $n=350$ within the time limit.

The differences in computational performance become particularly striking for the largest instances. For example, in the instance with $n=860$, LS computes a sparse generalized inverse in approximately one hour, whereas DRS fails to terminate within 24 hours. Similar behavior occurs for the instance with $n=1034$. These results reinforce the conclusion that local-search procedures currently provide the most scalable alternative for extremely large sparse least-squares problems.

However, we should emphasize that the numerical-quality comparison strongly favors DRS whenever computational resources allow its use. In several instances, the LS solutions exhibit 1-norms between $10\%$ and $200\%$ larger than those obtained by DRS. As large 1-norms may lead to numerical instability and amplification of computational errors, these results suggest that DRS may provide more robust generalized inverses for applications in which numerical stability is critical.

Overall, the NYPA case study reveals a more complex trade-off between sparsity, numerical robustness, and scalability than that observed in the synthetic experiments. The local-search procedures are clearly superior from the perspective of computational scalability and remain highly competitive in terms of sparsity. However, the DRS approach consistently produces solutions with better numerical scaling, as reflected by their significantly smaller 1-norms.

Finally, the substantial gap between the behavior of the synthetic instances and the NYPA matrices suggests that the random test problems used in this work do not fully capture the structural complexity of these real least-squares applications. In particular, the fact that even LS fails to solve some instances within 24 hours indicates that the NYPA matrices are considerably more challenging than our synthetic benchmarks. This observation motivates future work aimed at developing more realistic random-instance generators and at designing scalable algorithms capable of simultaneously promoting sparsity and numerical robustness for generalized inverses arising in real-world applications. 

\begin{table}[H]
\centering 
\setlength{\tabcolsep}{5.5pt}
        \begin{tabular}{r|r|r|r|r|r|r|r|r|r}
         \multicolumn{2}{c}{}  &\multicolumn{2}{c|}{}&  \multicolumn{3}{c|}{\ref{prob:barB}(DRS)} &  \multicolumn{3}{c}{\ref{prob:barB}(LS)} \\ [4pt]     
        \multicolumn{1}{c|}{$n,d$}  &\multicolumn{1}{c|}{$\|A\|_0$} &  \multicolumn{1}{c|}{$\|A^\dagger\|_0$} &  \multicolumn{1}{c|}{$\|A^\dagger\|_1$} &  \multicolumn{1}{c|}{$\|H\|_0$} &  \multicolumn{1}{c|}{$\|H\|_1$}&  \multicolumn{1}{c|}{Time}&  \multicolumn{1}{c|}{$\|H\|_0$} &  \multicolumn{1}{c|}{$\|H\|_1$}&  \multicolumn{1}{c}{Time}\\[4pt]
        \hline
        14\vphantom{$\frac{H^0}{H^0}$},~0.522 & 234 & 448 & 24.4 & 377 & 23.0 & 3.16 & 320 & 27.8 & 0.02 \\
        350,~0.022 & 4357 & 108376 & 2889.8 & 69437 & 2373.9 & 142.50 & 65895 & 3084.4 & 0.11 \\
        860,~0.013 & 18391 & 1279313 & 21154.0 & 713314 & 14301.3 & 1475.94 & 750352 & 39306.4 & 1.41 \\
        1034,~0.013 & 26719 & 1932327 & 109414259.4 & \multicolumn{1}{c|}{*} & \multicolumn{1}{c|}{*} & \multicolumn{1}{c|}{*} & 1172262 & 47696.5 & 2.07 \\
        3320,~0.006 & 93091 & 13160066 & 1404974.4 & \multicolumn{1}{c|}{*} & \multicolumn{1}{c|}{*} & \multicolumn{1}{c|}{*} & 6079289 & 2066437.7 & 127.73 \\
        \end{tabular}
    \caption{Computation of ah-symmetric reflexive generalized inverse for Maragal instances}
    \label{tab:lswithA_cs}
\end{table}

\begin{table}[H]
\centering    
\setlength{\tabcolsep}{5.5pt}
        \begin{tabular}{r|r|r|r|r|r|r|r|r}
         \multicolumn{1}{c}{}  &\multicolumn{2}{c|}{} &  \multicolumn{3}{c|}{\ref{probHhat1sym}(DRS)} &  \multicolumn{3}{c}{\ref{probHhat1sym}(LS)} \\ [4pt]     
        \multicolumn{1}{c|}{$n,d$}   &  \multicolumn{1}{c|}{$\|\hat{A}^\dagger\|_0$} &  \multicolumn{1}{c|}{$\|\hat{A}^\dagger\|_1$} &  \multicolumn{1}{c|}{$\|\hat{H}\|_0$} &  \multicolumn{1}{c|}{$\|\hat{H}\|_1$}&  \multicolumn{1}{c|}{Time}&  \multicolumn{1}{c|}{$\|\hat{H}\|_0$} &  \multicolumn{1}{c|}{$\|\hat{H}\|_1$}&  \multicolumn{1}{c}{Time}\\[4pt]
        \hline
            14\vphantom{$\frac{H^0}{H^0}$},~1.000 & 196 & 15.0 & 156 & 12.5 & 3.39 & 100 & 23.2 & 0.03 \\
            350,~0.147 & 63494 & 56698.2 & 26311 & 38549.6 & 2056.11 & 25049 & 68530.0 & 19.99 \\
            860,~0.551 & 733594 & 519084.9 & \multicolumn{1}{c|}{*} & \multicolumn{1}{c|}{*} & \multicolumn{1}{c|}{*} & 317557 & 1641794.7 & 3348.64 \\
            1034,~0.649 & 1052157 & 421938.6 & \multicolumn{1}{c|}{*} & \multicolumn{1}{c|}{*} & \multicolumn{1}{c|}{*} & 559545 & 3290177.1 & 1600.72 \\
            3320,~0.437 & 10895582 & 13310541498.8 & \multicolumn{1}{c|}{*} & \multicolumn{1}{c|}{*} & \multicolumn{1}{c|}{*} & \multicolumn{1}{c|}{*} & \multicolumn{1}{c|}{*} & \multicolumn{1}{c}{*} \\
        \end{tabular}
    \caption{Computation of symmetric generalized inverse for Maragal instances}
    \label{tab:lswithhatA_cs}
\end{table}


\section*{Acknowledgments}
 The local-search procedures were implemented by Gabriel Ponte, who kindly provided his code for our experiments.

	\bibliographystyle{ieeetr}
    \bibliography{biblio}
    
\end{document}